\documentclass[a4paper,12pt]{amsart}
\usepackage{amsmath}
\usepackage{amssymb}
\usepackage{mathrsfs}
\usepackage{enumerate}
\usepackage{ifthen}
\usepackage{graphicx}
\usepackage{caption}
\usepackage[T1]{fontenc} 
\usepackage{tikz}

\newcommand{\D}{\mathbb{D}}

\renewcommand{\H}{\mathcal{H}}

\renewcommand{\ss}{\mathcal{S}}

\nonstopmode \numberwithin{equation}{section}
\newtheorem{theorem}{Theorem}[section]

\newtheorem{lemma}{Lemma}[section]
\newtheorem{problem}{Problem}[section]

\theoremstyle{remark}
\theoremstyle{definition}
\newtheorem{remark}{Remark}[section]

\theoremstyle{plain}
\numberwithin{equation}{section}
\numberwithin{theorem}{section}

\newtheorem{innercustomthm}{Theorem}
\newenvironment{customthm}[1]
  {\renewcommand\theinnercustomthm{#1}\innercustomthm}
  {\endinnercustomthm}

\begin{document}

\title{Variability regions for Schur class}

\author{Chen Gangqiang}
\address{Chen Gangqiang,
School of Mathematics and Computer Sciences,
Nanchang University, Nanchang 330031, China}
\address{
Graduate School of Information Sciences, Tohoku University, Aoba-ku, Sendai 980-8579, Japan}
\email{cgqmath@qq.com; chenmath@ncu.edu.cn}
\date{\today}


\begin{abstract}
Let ${\mathcal S}$ be the class of analytic functions $f$ in the unit disk ${\mathbb D}$ with $f({\mathbb D}) \subset \overline{\mathbb D}$.
Fix pairwise distinct points $z_1,\ldots,z_{n+1}\in \mathbb{D}$ and corresponding interpolation values $w_1,\ldots,w_{n+1}\in \overline{\mathbb{D}}$. Suppose that $f\in{\mathcal S}$ and $f(z_j)=w_j$, $j=1,\ldots,n+1$. Then for each fixed $z \in {\mathbb D} \backslash \{z_1,\ldots,z_{n+1} \}$, we obtained a multi-point Schwarz-Pick Lemma, which determines the region of values of $f(z)$.

Using an improved Schur algorithm in terms of hyperbolic divided differences,
we solve a Schur interpolation problem involving a fixed point together with the hyperbolic derivatives up to a certain order at the point, which leads to a new interpretation to a generalized Rogosinski's Lemma.
For each fixed $z_0 \in {\mathbb D}$, $j=1,2, \ldots n$ and $\gamma  = (\gamma_0, \gamma_1 , \ldots , \gamma_n) \in {\mathbb D}^{n+1}$, denote by $H^jf(z)$ the hyperbolic derivative of order $j$ of $f$ at the point $z\in {\mathbb D}$,
 let ${\mathcal S} (\gamma) = \{f \in {\mathcal S} : f (z_0) = \gamma_0,H^1f (z_0) = \gamma_1,\ldots ,H^nf (z_0) = \gamma_n \}$.
We  determine the region of variability $V(z, \gamma ) = \{ f(z) : f \in {\mathcal S} (\gamma) \}$ for $z\in {\mathbb D} \backslash \{ z_0 \}$, which can be called "the generalized Rogosinski-Pick Lemma for higher-order hyperbolic derivatives".
\end{abstract}

\thanks{}

\maketitle
\pagestyle{myheadings}
\markboth{Chen Gangqiang}{Variability regions for Schur class}


\section{Introduction and preliminaries}
We denote by ${\mathbb C}$ the complex plane. Let ${\mathbb D}(c,\rho) = \{ z \in {\mathbb C} : |z-c| < \rho \}$,
and $\overline{\mathbb D}(c,\rho) = \{ z \in {\mathbb C} : |z-c| \leq \rho \}$
for $c \in {\mathbb C}$ and $\rho > 0$.
In particular, we denote the open and closed unit disks ${\mathbb D}(0,1)$ and
$\overline{{\mathbb D}}(0,1)$
by ${\mathbb D}$ and $\overline{{\mathbb D}}$ respectively.
From now on,  $\H$ denotes the class of analytic functions of $\D$ into itself, and $\mathcal{S}$  denotes the Schur class, i.e., the set of analytic functions $f$ in ${\mathbb D}$ with $| f| \leq 1$.
The  main  goal  of  this  paper  is  to  present  a  new  theory  which  essentially  concerns  the  Schur class.

The classic Schwarz-Pick Lemma \cite{Pick1916} states that $$\left|\frac{f(z)-f(z_1)}{1-\overline{f(z_1)}f(z)}\right|\le \left|\frac{z-z_1}{1-\overline{z_{1}} z}\right|,\quad f\in\H,\quad z_1\in \D.$$
Nontrivial equality holds if and only if $f$ is a conformal automorphism of $\D$ (see also \cite{avkhadiev2009schwarz}).
We denote by $\text{Aut}(\D)$ the group of conformal automorphisms of $\D$.
If we define $T_{a}\in \text{Aut}(\D)$ by $T_{a}(z):=(z+a) /(1+\bar{a} z)$ for $z,a\in \mathbb D$,
then the Schwarz-Pick Lemma can be rewritten as follows.
\begin{customthm}{A (The Schwarz-Pick Lemma)}\label{thm:g-first}
Let $z_1,w_1\in\D$. Suppose that $f\in \H$, $f(z_1)=w_1$.
Set
$$
 f_{\alpha}(z) = T_{w_1}\big(\alpha T_{-z_1}(z)\big).$$
 Then for each fixed $z \in {\mathbb D} \setminus \{z_1\}$,
the region of values of $f(z)$ is the closed disk
$$\overline{\D}(c_1(z),\rho_1(z))
  = \{  f_{\alpha}(z ) : \alpha \in \overline{\mathbb D} \},
$$
where
$$c_1(z)=\frac{w_1(1-|T_{-z_1}(z)|^2)}{1-|w_1|^2|T_{-z_1}(z)|^2},\quad \rho_1(z)=\frac{(1-|w_1|^2)|T_{-z_1}(z)|}{1-|w_1|^2|T_{-z_1}(z)|^2},$$
and $f(z)$ is the form of
$$T_{w_1}\big(T_{-z_1}(z) g^*(z)\big),$$
where $g^*\in\ss$.
Further, $f(z)\in \partial\D(c_1(z),\rho_1(z))$ if and only if
$f=f_{\alpha}$ for some constant $\alpha \in \partial \D$.
\end{customthm}
Mercer \cite{mercer1997sharpened} and Kaptano{\u{g}}lu \cite{kaptanoglu2002refine} respectively concluded a two-point Schwarz-Pick Lemma in which  the images of two points in $\D$ are known. A number of sharpened versions
of the Schwarz-Pick Lemma can be found for example in  \cite{beardon2008multi-Dieudonne-points,chen2019,chen2020,
chen2021,cho2012multi}.
We will first extend their results to a multi-point Schwarz-Pick Lemma. For this purpose, we consider the following problem.
\begin{problem}\label{the-problem}
Let $n \in {\mathbb N}$. Fix pairwise distinct points $z_1,\ldots,z_{n+1}\in \mathbb{D}$ and corresponding interpolation values $w_1,\ldots,w_{n+1}\in \overline{\mathbb{D}}$. Suppose that $f\in\ss$ and $f(z_j)=w_j$, $j=1,\ldots,n+1$. Then for each fixed $z \in {\mathbb D} \backslash \{z_1,\ldots,z_{n+1} \}$, determine the region of values of $f(z)$.
\end{problem}
We  now  focus  on  another  famous  result  of  complex  analysis:  Rogosinski's  lemma \cite{rogosinski1934}. In  what  follows,
we  will  see  how  Rogosinski's  lemma  can  be  interpreted  in  terms  of  hyperbolic  derivatives.  Nevertheless, let's begin by recalling  the  statement  of  this  lemma.
\begin{customthm}{B (Rogosinski's Lemma)}
Let $f\in \H$ with $f(0)=0$ and $|f'(0)|<1$.
Then for $z\in\D\setminus\{0\}$, the region of values of $f(z)$ is the closed disk $\overline{\D}(c'(z),\rho'(z))$, where
 $$c'(z)=\dfrac{zf'(0)(1-|z|^2)}{1-|z|^2|f'(0)|^2},\quad \rho'(z)= \dfrac{|z|^2(1-|f'(0)|^2)}{1-|z|^2|f'(0)|^2}.$$
\end{customthm}
It is of interest to consider Rogosinski's Lemma as a
sharpened version of Schwarz's Lemma (see \cite{Duren1983univalent}).
In 2011, Rivard became interested again on this topic \cite{rivard2013application}. Avoiding the restriction $f(0) = 0$, he has been able to give an analogous version, called the Rogosinski-Pick Lemma, in terms of the hyperbolic derivative of first order.
Moreover,
his observation asserts that for $f\in\H$ and $z_0\in\D$,
depending on the hyperbolic derivative at $z_0$, we can determine the variability region of $f(z)$ for any point $z\in \D$.
Recently numerous significant articles concerning variability regions have been obtained \cite{Ali-Allu-Yanagihara-1,Ali-Allu-Yanagihara-2,Yanagihara2024Julia,Ponnusamy2007univalent,
Ponnusamy2009satisfying,
yanagihara2006Nachr,yanagihara2010families}.

The following text is devoted to the hyperbolic
divided differences and the higher-order hyperbolic derivatives. We will see their importance in
the solution to a Schur interpolation problem.
The notion of hyperbolic divided differences for $f\in \ss$ was recently introduced by Baribeau et al. \cite{Baribeau-Rivard2009} (see also \cite{Baribeau2013} and \cite{beardon2004multi}).
For $ z, w\in\overline{\D}$, we define
\begin{equation}\label{eq:dist}
[z,\,w]:=
\begin{cases}
\dfrac{z-w}{1-\overline{w}z} &\quad \text{if}~z\bar w\ne 1; \\
\infty &\quad\text{if}~z\bar w= 1.
\end{cases}
\end{equation}
It is convenient to memorize the fact that $T_a (z)=[z,-a]$ and $z=[[z,a],-a]$ for $z,a\in \mathbb D$.
We construct an operator $\Delta_{z_0}$, which maps every function $f\in \ss$ to $\Delta_{z_0} f\in \ss$, by
\begin{equation}\label{eq:hdquo}
\Delta_{z_0}f(z)=
\begin{cases}
\dfrac{[f(z),\,f(z_0)]}{[z,\,z_0]} &\quad \text{for}~z\ne z_0, \\
\null & \\
\dfrac{(1-|z_0|^2)f'(z_0)}{1-|f(z_0)|^2} &\quad\text{for}~z=z_0.
\end{cases}
\end{equation}
For $f,g\in\H$ and $z_{0}\in \mathbb{D}$ , we have the chain rule directly
$$
\Delta_{z_0}(f\circ g)=(\Delta_{g(z_0)}f)\circ g\cdot\Delta_{z_0}g.
$$
Let functions
 $f\in \mathcal{S} $ and fix $k$ pairwise distinct points $z_1, \ldots , z_k\in \mathbb{D}$. Set
$$\Delta^0f(z):=f(z),\quad z\in\mathbb{D},$$
Then we can iterate the process and construct the hyperbolic divided difference of order $j$ of the function $f$ for distinct parameters $z_1,\cdots,z_{j}$ as follows (cf. \cite{Baribeau-Rivard2009}):
$$
\Delta^jf(z;z_{j},\cdots,z_1)=(\Delta_{z_{j}}\circ\cdots\circ\Delta_{z_1})f(z),
$$
or recursively defined by
\begin{equation}
\Delta^jf(z;z_j,\cdots,z_1)=\frac{[\Delta^{j-1}f(z;z_{j-1},\cdots,z_1),\Delta^{j-1}f(z_j;z_{j-1},\cdots,z_1)]}{[z,z_j]}.
\end{equation}
The  concept  of  hyperbolic  divided  differences  involves  parameters  chosen  from  the  unit  disk, all of which are distinct  from  each  other.
In  what  follows,  we  will  deal  with  the  same  definition with only one single parameter (see \cite{rivard2011higher-orderHyperbolicDerivatives}).
Using a limiting process, we can define the hyperbolic divided difference of order $n$ of $f$ with parameter $z$ by
\begin{equation}\label{def:Delta-z-f-zeta}
\Delta_z^n f(\zeta)=\Delta^n f(\zeta;z,\cdots,z):=\lim_{z_{n}\to z}\cdots\lim_{z_1\to z}\Delta^n f(\zeta;z_{n},\cdots,z_1)\quad(\zeta\neq z).
\end{equation}
Then we can define the $n$-th order  hyperbolic derivative  of $f\in \H$ at the point $z\in \D$ by
$$H^n f(z):=\Delta_z^n f(z)=\Delta^n f(z;z,\cdots,z):=\lim_{\zeta\to z}\Delta^n f(\zeta;z,\cdots,z).$$
We note that the usual hyperbolic derivative coincides with the first-order hyperbolic derivative $H^{1} f$,
$$f^h(z):=\frac{(1-|z|^2)f'(z)}{1-|f(z)|^2}=H^{1} f(z).$$

In the following text, we shall present the essential tools which lead to a generalization of the Schwarz-Pick Lemma for hyperbolic divided differences and can be used to demonstrate a generalization of Rogosinski's Lemma for higher-order hyperbolic derivatives.
For given pairwise distinct points $z_1,\ldots,z_{n+1}\in\mathbb{D}$ and interpolation values $w_1,\ldots,w_{n+1}\in\overline{\mathbb{D}}$ we define the hyperbolic divided differences
$
\Delta_j^k$, $k=0,\ldots,n$, $j=k+1,\ldots,n+1$, by the following recursive procedure: Let $\Delta_j^0: = w_j$ for $j= 1, \ldots , n+1.$ Assume that, for some $k$ with $1\le k\le n$ and $j= k+ 1, \ldots , n+1$,  the divided differences $\Delta_{j}^{k-1}$ and $\Delta_{k}^{k-1}$ are given. Then
\begin{align}
&\Delta_{j}^{k}:=\frac{\left[\Delta_j^{k-1},\Delta_k^{k-1}\right]}{\left[z_j,z_k\right]},
&&\mathrm{if}\quad\left|[\Delta_j^{k-1},\Delta_k^{k-1}]\right|\leq\left|[z_j,z_k]\right|, \\
&\Delta_{j}^{k} :=\infty, &&\mathrm{in~all~other~cases.}
\end{align}
The essence of the rules is that the first formula is used as long as
the modulus of the result does not exceed 1,
otherwise $\Delta_j^k$ is set to $\infty$,
only except that $\Delta_j^k:=0$ if $\Delta_j^{k-1} =\Delta_k^{k-1}\in \partial\D$.

Since the hyperbolic divided differences lead naturally to the concept of higher-order hyperbolic derivatives,
we can treat a generalized version of Rogosinski's Lemma in a hyperbolic manner, which concerns the hyperbolic derivatives of order $n\ge 1$. That is to say,
we consider the Schur class where the first $n$-th hyperbolic derivatives at a fixed point are prescribed.
 To this end, we need to determine a solution to a particular interpolation problem involving a fixed point together with the hyperbolic derivatives up to a certain order at the point. Thus we  consider  the following two problems.
\begin{problem}[The Schur interpolation problem]\label{prob:Schur}
Let $n \in {\mathbb N}$, $z_0 \in {\mathbb D}$ and $\gamma  = (\gamma_0, \ldots , \gamma_n ) \in \overline{{\mathbb D}}^{n+1}$. Determine all functions $f$ satisfying
$$
{\mathcal S} (\gamma ) =
   \left\{ f \in {\mathcal S} : \; f(z_0)=\gamma_0,H^1f(z_0)=\gamma_1,
   \cdots H^nf(z_0)=\gamma_n   \right\},
$$
\end{problem}
For the sake of completeness, we give a solution to
the Schur interpolation problem \ref{prob:Schur}, which plays an important role
 in a new manner Rogosinski's Lemma.
\begin{problem}\label{the_problem}
For $z \in {\mathbb D} \backslash \{ z_0 \}$,
describe the variability region
\begin{equation*}
 V(z, \gamma )
 =
 \left\{f(z)
   : \; f \in {\mathcal S}  (\gamma )
  \right\}.
\end{equation*}
\end{problem}
 We note that Schur \cite{Schur-1917} and Ali et al. \cite{Ali-Allu-Yanagihara-1} have characterized Problem \ref{prob:Schur} and Problem \ref{the_problem} for $z_0=0$, respectively.  We will extend their results to $z_0\in \D$.

The organization of this paper is as follows. In Section 2, we state our first main theorem, called the multi-point Schwarz-Pick Lemma, which gives the solution to Problem \ref{the-problem}. We also provide several lemmas on the rational functions associated with the hyperbolic divided differences and give the detailed proof of our main theorem.
In Section 3, we reformulate the Schur algorithm in terms of hyperbolic divided differences, which allow us to solve the Schur interpolation problem. We also give an analogous version of Rogosinski's Lemma, called the generalized Rogosinski-Pick Lemma for higher-order hyperbolic derivatives, which gives the solution to Problem \ref{the_problem}.

\section{Multi-point Schwarz-Pick Lemma}
Throughout this section, we will assume that all the hyperbolic divided differences $|\Delta_j^{k}|<1$ for $0\le k<j\le n+1$, unless otherwise stated.
Schur \cite{Schur-1917} developed a recursive procedure to solve an interpolation problem.
In our study, we rephrase the Schur algorithm in terms of hyperbolic divided differences to construct all solutions $f\in\ss$ such that  $f(z_j)=w_j$ for $j=1,\ldots,n+1$.
Table \ref{table:Delta} describes the hyperbolic divided differences and the solutions $f$. It is worth mentioning that the algorithm used here can also be applied to variability regions of certain analytic functions \cite{Ali-Allu-Yanagihara-1}.

\begin{table}[h]
\caption{Table of hyperbolic divided differences for the multi-point Schwarz-Pick Lemma}
\label{table:Delta}
\begin{tabular}{|c|c|cccccccc}
\hline \multicolumn{2}{|c|}{ Points of $\D$ } & 1 & 2 & 3 & $\ldots$ & $n-1$ & $n$&$n+1$& \\
\hline$z_{1}$ & $w_{1}=\Delta_{1}^{0}$ & & & & & &\\
& &$\Delta_{2}^{1}$ & & & &
\\$z_{2}$ & $w_{2}=\Delta_{2}^{0}$ &  & $\Delta_{3}^{2}$& & & &
\\& &$\Delta_{3}^{1}$  & &  $\Delta_{4}^{3}$& &
 \\$z_{3}$ & $w_{3}=\Delta_{3}^{0}$   & &$\Delta_{4}^{2}$  & &$\ddots$ &
 \\&&$\Delta_{4}^{1}$ &  &$\Delta_{5}^{3}$ & & $\Delta_{n}^{n-1}$ &
\\$z_{4}$ & $w_{4}=\Delta_{4}^{0}$ & &$\Delta_{5}^{2}$ &  & $\vdots$&&$\Delta_{n+1}^{n}$&& \\& & $\Delta_{5}^{1}$ & & $\vdots$ & & $\Delta_{n+1}^{n-1}$ &&$\Delta^{n+1}f(z)$
 \\$z_{5}$ & $w_{5}=\Delta_{5}^{0}$ & & $\vdots$ & &$.\cdot$ & &$\Delta^nf(z)$& \\$\vdots$ & $\vdots$ & $\vdots$ & & $\Delta_{n+1}^{3}$ & &$\Delta^{n-1}f(z)$ & \\$z_{n}$ & $w_{n}=\Delta_{n}^{0}$ &  & $\Delta_{n+1}^{2}$ & &$\cdot^{\cdot^{\cdot}}$ & &  \\&&$\Delta_{n+1}^{1}$ & &$\Delta^3f(z)$ & &
\\$z_{n+1}$ & $w_{n+1}=\Delta_{n+1}^{0}$ & &$\Delta^2f(z)$ & & & & \\
&&$\Delta^1f(z)$&&&
\\$z$ & $\Delta^0f(z)=f(z)$ & & & & & &
\\\hline\end{tabular}
\end{table}

We first determine all functions $f\in \ss$ satisfying $f(z_1)=w_1=\Delta_1^0$.
By the maximum modulus principle, if $|\Delta_1^0|>1$, then there is no $f\in \ss$ such that $f(z_1)=\Delta_1^0$; if $|\Delta_1^0|=1$, then there is a unique solution
$f (z) \equiv \Delta_1^0$. Suppose  that  $|\Delta_1^0| <1$.
Then
$$
   \Delta^1f(z)
   =
   \frac{[f (z),f(z_0)]}{[z,z_0]}
$$
belongs to $\mathcal{S}$.
Consequently, the set of all solutions $f \in \mathcal{S}$ is given by
$$
   f (z)
   = T_{\Delta_1^0}(T_{-z_0}(z) \Delta^1f (z))
   = \frac{T_{-z_0}(z) \Delta^1f (z) + \Delta_1^0}
   {1+\overline{\Delta_1^0}T_{-z_0}(z) \Delta^1f (z)},
$$
where $\Delta^1f \in \mathcal{S}$ is arbitrary.
By recursively applying this procedure, we can obtain all general solutions $f\in\ss$ such that  $f(z_j)=w_j$ for $j=1,\ldots,n+1$.
Starting with $\Delta^0f:=f$, we define $\Delta^1f, \ldots , \Delta^{n+1}f
\in \mathcal S$ by
\begin{equation}\label{eq:Def_f_k}
\Delta^{k}f(z) = \frac{[\Delta^{k-1}f(z) ,\Delta^{k-1}f(z_k)]}{[z,z_k]}.
\end{equation}
We claim that
\begin{equation}\label{eq:Delta-k-f}
\Delta^kf(z_j)=\Delta_j^k,\quad j=k+1,\ldots, n.
\end{equation}

Indeed, for $k=0$, it is true that $f_0(z_j)=z_j$ for $j=1,\ldots,n$.
Assuming now that the assertion \eqref{eq:Delta-k-f} is true for some $k-1$ and all $j=k,\ldots,n$, then we have
$$
\Delta^kf(z_j)=\frac{[\Delta^{k-1}f(z_j),\Delta^{k-1}f(z_k)]}{[z_j,z_k]}
=\frac{[\Delta_j^{k-1},\Delta_k^{k-1}]}{[z_j,z_k]}=\Delta_j^k.
$$
In particular, $\Delta^kf(z_{k+1})=\Delta_{k+1}^k$.

The idea is to find $f$ from $\Delta^{n+1}f$ using the inversion formula from \eqref{eq:Def_f_k}.
Therefore, starting with $f_{n+1}\in \mathcal{S}$ we define the functions $f_n,\ldots,f_0$ recursively by
\begin{equation}\label{eq:f-k-chain}
f_{k}(z):=\big[[z,z_{k+1}]\cdot f_{k+1}(z),-\Delta_{k+1}^{k}\big]=T_{\Delta_{k+1}^{k}}(T_{-z_{k+1}}(z)f_{k+1}(z)).
\end{equation}
Then
$$f_k( z_{k+1}) = \big[[z_{k+1},z_{k+1}]\cdot f_{k+1}(z),-\Delta_{k+1}^{k}\big]=[ 0, -\Delta_{k+1}^{k}] = \Delta_{k+1}^{k},$$
and therefore $\Delta_{z_{k+1}}f_k= f_{k+ 1}$.

We proof by induction from $k=n$ to $k=0$ that
\begin{equation}\label{eq:f-k-Delta}
f_k(z_j)=\Delta_j^k, \quad j=k+1,\ldots,n+1.
\end{equation}
Indeed, for $k=n$, we already have $f_n( z_{n+1})=\Delta_{n+1}^{n}$. Assume \eqref{eq:f-k-Delta} holds for some $k$ and all $j=k+1,\ldots,n+1$, then for $j$ we have
$$
\begin{aligned}
f_{k-1}(z_{j})& =\begin{bmatrix}[z_j,z_k]\cdot f_k(z_j),-\Delta_k^{k-1}\end{bmatrix}  \\
&=\begin{bmatrix}[z_j,z_k]\cdot\Delta_j^k,-\Delta_k^{k-1}\end{bmatrix} \\
&=\left[[z_j,z_k]\cdot\frac{\left[\Delta_j^{k-1},\Delta_k^{k-1}\right]}{[z_j,z_k]},
-\Delta_k^{k-1}\right] \\
&=\left[[\Delta_j^{k-1},\Delta_k^{k-1}],-\Delta_k^{k-1}\right]=\Delta_j^{k-1}.
\end{aligned}$$
Together with $f_{k-1}( z_k)=\Delta_k^{k-1}$, we prove \eqref{eq:f-k-Delta}.
In particular, $f_0( z_j) = \Delta_j^0= w_j$ for $j= 1, \ldots , n+1$. We now set
$f:=f_0$, then $\Delta^kf(z)=f_k(z)$
so that $f=f_0$ is the form of
$$
f (z) =
      T_{\Delta_1^0} ( T_{-z_1}(z) T_{\Delta_2^1} ( \cdots T_{-z_n}(z) T_{\Delta_{n+1}^{n}} ( T_{-z_{n+1}}(z)f_{n+1}(z)) \cdots )),
$$
where $f_{n+1} \in \mathcal{S}$.
We rewrite \eqref{eq:f-k-chain} as
\begin{equation}\label{eq:recurrence_formula_for_f-0}
f (z ) =  f_0 (z)=\frac{T_{-z_{1}}(z) f_{1} (z)+ \Delta_{1}^0 }{1+ \overline{\Delta_{1}^0 }T_{-z_{1}}(z) f_{1} (z)},
\end{equation}
and
\begin{equation}\label{eq:recurrence_formula_for_f-k}
f_k (z ) = \frac{T_{-z_{k+1}}(z) f_{k+1} (z)+ \Delta_{k+1}^k }{1+ \overline{\Delta_{k+1}^k }T_{-z_{k+1}}(z) f_{k+1} (z)},\quad (k=0,1,\ldots,n).
\end{equation}
We now define sequences of rational functions recursively by
\begin{equation}\label{eq:initial-A-B-0}
 \begin{pmatrix}
  A_0 (z) & \widetilde{A}_0(z) \\
  B_0 (z) & \widetilde{B}_0 (z) \\
 \end{pmatrix}
 =
 \begin{pmatrix}
  \overline{\Delta_1^0} & 1 \\
  1 & \Delta_1^0 \\
 \end{pmatrix}
\end{equation}
and
\begin{equation}\label{eq:A-B-k-recurrence-formula}
 \begin{pmatrix}
  A_{k+1} (z) & \widetilde{A}_{k+1}(z) \\
  B_{k+1} (z) & \widetilde{B}_{k+1} (z) \\
 \end{pmatrix}
 =
 \begin{pmatrix}
  T_{-z_{k+1}}(z) & \overline{\Delta_{k+2}^{k+1}} \\
  T_{-z_{k+1}}(z)\Delta_{k+2}^{k+1} & 1 \\
 \end{pmatrix}
 \begin{pmatrix}
  A_k (z) & \widetilde{A}_k (z) \\
  B_k (z) & \widetilde{B}_k (z) \\
 \end{pmatrix} ,
\end{equation}
where $k=0,\ldots , n-1 $. From \eqref{eq:initial-A-B-0} and
\eqref{eq:A-B-k-recurrence-formula} it easily follows that
\begin{equation}\label{eq:A-B-1}
 \begin{pmatrix}
  A_1 (z) & \widetilde{A}_1(z) \\
  B_1 (z) & \widetilde{B}_1 (z) \\
 \end{pmatrix}
 =
 \begin{pmatrix}
  \overline{\Delta_2^{1}}+\overline{\Delta_1^{0}}  T_{-z_{1}}(z) &
  \Delta_1^{0} \overline{\Delta_2^{1}}+T_{-z_1}(z) \\
  1+\overline{\Delta_1^{0}}\Delta_2^{1}T_{-z_{1}}(z) &
  \Delta_1^{0}+\Delta_2^{1}  T_{-z_{1}}(z)  \\
 \end{pmatrix}.
\end{equation}

By (\ref{eq:A-B-k-recurrence-formula}) and induction, we have the following recurrence formula for $f(z)$,
\begin{equation}\label{eq:recurrence_formula_for-f}
      f (z) =
      \frac{T_{-z_{k+1}}(z) \widetilde{A}_k(z) f_{k+1} (z) + \widetilde{B}_k(z)}
      {T_{-z_{k+1}}(z)A_k(z) f_{k+1} (z) + B_k(z)},\quad  k=0,1, \ldots , n.
\end{equation}

For $ \varepsilon \in \overline{\mathbb D}$, let
\begin{align}
f_{\varepsilon }(z)
    =&  T_{\Delta_1^0} ( T_{-z_1}(z) T_{\Delta_2^1} ( \cdots T_{-z_n}(z) T_{\Delta_{n+1}^{n}} ( \varepsilon T_{-z_{n+1}}(z)) \cdots )), \quad z \in {\mathbb D}.
\label{def:extremal-f}
\end{align}
 Then $f_{\varepsilon}\in \ss$ satisfies $f_{\varepsilon}(z_j)=\Delta_j^0=w_j$ for $1\le j\le n+1$.
 We note that for each fixed $\varepsilon \in \overline{\mathbb D}$, $f_{\varepsilon} (z)$ is an analytic function of $z \in {\mathbb D}$, and for each fixed $z \in {\mathbb D}$, $f_{\varepsilon} (z)$ is an analytic function of $\varepsilon \in \overline{\mathbb D}$.
In particular, when $\varepsilon \in \partial {\mathbb D}$, $f_{\varepsilon} (z)$ is a finite Blaschke product of $z$.  Here we recall that a (finite) Blaschke product of degree $n$ takes the form
         $$ B(z)=e^{i \theta}\prod\limits_{j=1}^{n}
         \frac{z-z_j}{1-\overline{z_j}z}, \quad \theta \in \mathbb{R},\quad z_1,\ldots,z_n\in \D.$$

 The following lemma follows from  \eqref{eq:recurrence_formula_for-f} with the notation $f^*=f_{n+1}$.
\begin{lemma}\label{lemma:Schur_polynomial_and_Caratheodoy_problem}
Fix pairwise distinct points $z_1,\ldots,z_{n+1}\in \mathbb{D}$ and corresponding interpolation values $w_1,\ldots,w_{n+1}\in \overline{\mathbb{D}}$. Let
all the hyperbolic divided differences
$\Delta_j^k \in {\mathbb D}$ for $0\le k<j\le n+1$.
Then for any
$f\in\ss$ such that
$f(z_j)=w_j$ for $1\le j\le n+1$, there exists a unique $f^{*}\in \mathcal{S}$ such that
\begin{equation}\label{eq:recurrence_representation_for_f}
f (z) = \frac{T_{-z_{n+1}}(z)\widetilde{A}_n(z) f^*(z)+ \widetilde{B}_n(z)}{T_{-z_{n+1}}(z)A_n(z)f^*(z) + B_n(z)} .
\end{equation}
Conversely, for any $f^{*}\in \mathcal S$, if  $f$ is given by $(\ref{eq:recurrence_representation_for_f})$, then $f$ satisfies $f(z_j)=w_j$ for $1\le j\le n+1$. In particular, the function $f_{\varepsilon}$ defined by (\ref{def:extremal-f}) can be written as
\begin{equation}\label{eq:extremal-f-varepsilon}
f_{\varepsilon } (z)
   = \frac{ \varepsilon T_{-z_{n+1}}(z)\widetilde{A}_n(z)+ \widetilde{B}_n(z)}{ \varepsilon T_{-z_{n+1}}(z) A_n(z) + B_n(z)},
\end{equation}
and $f_{\varepsilon}(z_j)=w_j$ for $1\le j\le n+1$.
\end{lemma}

We next give some important properties of the rational functions $A_k (z)$, $B_k (z)$, $ \widetilde{A}_k (z)$ and $\widetilde{B}_k(z)$ ($k=0,1, \ldots , n$), which
are of degree at most $k$.
We observe that $\widetilde{A}_0 (z)=\overline{B_0 (1/\overline{z})}$ and $\widetilde{B}_0 (z)=\overline{A_0 (1/\overline{z})}$. Generally, we have the following formulas.
\begin{lemma}\label{lemma:relation_between_A-k_B-k}
For $k=1, \ldots , n$,
\begin{equation}\label{eq:tildeAk_tildeBk}
\widetilde{A}_k (z) = \overline{B_k (1/\overline{z})}\prod_{\ell=1}^kT_{-z_{\ell}}(z),
   \quad
   \widetilde{B}_k (z) = \overline{A_k (1/\overline{z})}\prod_{\ell=1}^kT_{-z_{\ell}}(z),
\end{equation}
and
\begin{equation}\label{eq:Ak_Bk}
{A}_k (z) = \overline{\widetilde{B}_k (1/\overline{z})}\prod_{\ell=1}^kT_{-z_{\ell}}(z),
   \quad
   B_k (z) = \overline{\widetilde{A}_k (1/\overline{z})}\prod_{\ell=1}^kT_{-z_{\ell}}(z).
\end{equation}
\end{lemma}
\begin{proof}
We use induction on $k$ to prove \eqref{eq:tildeAk_tildeBk}.
When $k=1$,
\begin{align*}
T_{-z_1}(z) \overline{B_1 (1/\overline{z})}&=T_{-z_1}(z)(1+\Delta_1^{0}\overline{\Delta_2^{1}}\overline{T_{-z_{1}}(1/\overline{z})})\\
&=T_{-z_1}(z)+\Delta_1^{0}\overline{\Delta_2^{1}}(\overline{T_{-z_{1}}(1/\overline{z})}T_{-z_1}(z))\\
&=T_{-z_1}(z)+\Delta_1^{0}\overline{\Delta_2^{1}}= \widetilde{A}_1 (z).
\end{align*}
The second last equality comes from
$\overline{T_{-z_1}(1/ \overline {z})}=1/ T_{-z_1}(z)$.
Similarly we have
\begin{align*}
T_{-z_1}(z) \overline{A_1 (1/\overline{z})}&=T_{-z_1}(z)(\Delta_2^{1}+\Delta_1^{0}\overline{T_{-z_{1}}(1/\overline{z})})\\
&=\Delta_2^{1}T_{-z_1}(z)+\Delta_1^{0}(\overline{T_{-z_{1}}(1/\overline{z})}T_{-z_1}(z))\\
&=\Delta_2^{1}T_{-z_1}(z)+\Delta_1^{0}= \widetilde{B}_1 (z).
\end{align*}
Assume that \eqref{eq:tildeAk_tildeBk} holds for $k \geq 1$,
then by \eqref{eq:A-B-k-recurrence-formula} we have
\begin{align*}
\widetilde{A}_{k+1} (z)
     &=
     T_{-z_{k+1}}(z) \widetilde{A}_k (z) + \overline{\Delta_{k+2}^{k+1}} \widetilde{B}_k (z)\\
     &=
     \prod_{\ell=1}^{k+1}T_{-z_{\ell}}(z) \overline{B_k(1/\overline{z})} + \overline{\Delta_{k+2}^{k+1}} \prod_{\ell=1}^{k}T_{-z_{\ell}}(z)  \overline{A_k(1/ \overline{z})}\\
     &=
    \prod_{\ell=1}^{k+1}T_{-z_{\ell}}(z)
     \left\{\overline{\Delta_{k+2}^{k+1}(1/ \overline{T_{-z_{k+1}}(z)}) A_k(1/ \overline{z}) + B_k(1/ \overline{z})}\right\}\\
        &=
      \prod_{\ell=1}^{k+1}T_{-z_{\ell}}(z)
     \left\{\overline{\Delta_{k+2}^{k+1}T_{-z_{k+1}}(1/ \overline{z}) A_k(1/ \overline{z}) + B_k(1/ \overline{z})}\right\}
     =
    \prod_{\ell=1}^{k+1}T_{-z_{\ell}}(z) \overline{B_{k+1} (1/ \overline{z})} .
\end{align*}
Similarly $\widetilde{B}_{k+1} (z) = \prod_{\ell=1}^{k+1}T_{-z_{\ell}}(z)\overline{A_{k+1} (1/ \overline{z})}$.

From
$\widetilde{A}_k (z) = \overline{B_k (1/\overline{z})}\prod_{\ell=1}^kT_{-z_{\ell}}(z)$, we can get
$\widetilde{A}_k (1/\overline{z}) = \overline{B_k (z)}\prod_{\ell=1}^kT_{-z_{\ell}}(1/\overline{z})$, and then
$$\overline{\widetilde{A}_k (1/\overline{z})}
=B_k (z)\prod_{\ell=1}^k\overline{T_{-z_{\ell}}(1/\overline{z})}
=\frac{B_k (z)}{\prod_{\ell=1}^kT_{-z_{\ell}}(z)},$$
which yields $B_k (z) = \overline{\widetilde{A}_k (1/\overline{z})}\prod_{\ell=1}^kT_{-z_{\ell}}(z)$.
Similarly ${A}_k (z) = \overline{\widetilde{B}_k (1/\overline{z})}\prod_{\ell=1}^kT_{-z_{\ell}}(z)$.
\end{proof}

It  follows directly from \eqref{eq:initial-A-B-0} that $\widetilde{A}_0 (z) B_0(z) - A_0(z)  \widetilde{B}_0(z)=1-|\Delta_1^0|^2$. Generally, the following result will appear.
\begin{lemma}\label{lemma:tilde-A-B-k}
For $k=1, \ldots , n$,
\begin{equation}\label{eq:tilde-A-B-k}
  \widetilde{A}_k (z) B_k(z) - A_k(z)  \widetilde{B}_k(z) =  \prod_{\ell=1}^k T_{-z_{\ell}}(z) \prod_{\ell=0}^k (1-|\Delta_{\ell+1}^{\ell}|^2) .
\end{equation}
\end{lemma}
\begin{proof}
For $k \geq 1$, we can prove the formulas elegantly by means of matrix calculus. We start from the equation
\begin{equation}\label{eq:induction}
 \begin{pmatrix}
  A_k (z) & \widetilde{A}_k(z) \\
  B_k (z) & \widetilde{B}_k (z) \\
 \end{pmatrix}
 =
  \begin{pmatrix}
  T_{-z_k}(z) & \overline{\Delta_{k+1}^{k}} \\
  \Delta_{k+1}^{k}T_{-z_k}(z) & 1 \\
 \end{pmatrix}
 \cdots
  \begin{pmatrix}
  T_{-z_1}(z) & \overline{\Delta_2^{1}} \\
  \Delta_2^{1}T_{-z_1}(z) & 1 \\
 \end{pmatrix}
 \begin{pmatrix}
  \overline{\Delta_1^0} & 1 \\
  1 & \Delta_1^0 \\
 \end{pmatrix}
 ,
\end{equation}
which follows easily from the recursive formula \eqref{eq:A-B-k-recurrence-formula}. Taking the determinant on both sides of the equation, we get \eqref{eq:tilde-A-B-k}.
\end{proof}

Because of \eqref{eq:tildeAk_tildeBk} and \eqref{eq:Ak_Bk}, Lemma \ref{lemma:tilde-A-B-k} can also be written as follows.
\begin{lemma}
For $k=0, \ldots , n$,
\begin{align*}
&\widetilde{A}_k (z) \overline{\widetilde{A}_k (1/\overline{z})}
- A_k(z)  \overline{A_k (1/\overline{z})} = \prod_{\ell=0}^k (1-|\Delta_{\ell+1}^{\ell}|^2),\\
&B_k(z)\overline{B_k (1/\overline{z})} -  \widetilde{B}_k(z)\overline{\widetilde{B}_k (1/\overline{z})} = \prod_{\ell=0}^k (1-|\Delta_{\ell+1}^{\ell}|^2).
\end{align*}
\end{lemma}
\begin{lemma}\label{lemma:tildeAK+1-Ak}
For $k=1, \ldots , n$,
\begin{equation}\label{eq:tildeAK+1-Ak}
  \widetilde{A}_k(z)A_{k+1}(z)  -A_{k}(z) \widetilde{A}_{k+1}(z)=
  \overline{\Delta_{k+2}^{k+1}}\prod_{\ell=1}^kT_{-z_{\ell+1}}(z)\prod_{\ell=0}^k (1-|\Delta_{\ell+1}^{\ell}|^2).
\end{equation}
In particular,
$$A_1(z)  \widetilde{A}_0(z)-A_0(z) \widetilde{A}_1(z)
=\overline{\Delta_2^1}(1-|\Delta_1^0|^2).$$
\end{lemma}
\begin{proof}
When $k=0$, from \eqref{eq:initial-A-B-0} and \eqref{eq:A-B-1}, we have
\begin{align*}
A_{1}(z) \widetilde{A}_{0}(z)-A_{0}(z) \widetilde{A}_{1}(z)
&=\{\overline{\Delta_2^{1}}+\overline{\Delta_1^{0}} T_{-z_{1}}(z)\}
-\overline{\Delta_1^{0}}\{\Delta_1^{0} \overline{\Delta_2^{1}}+T_{-z_1}(z)\}\\
&=\overline{\Delta_2^{1}}(1-|\Delta_1^0|^2).
\end{align*}

For $k\ge 1$, by Lemma \ref{lemma:tilde-A-B-k} we have
\begin{align*}
&\widetilde{A}_k(z)A_{k+1}(z)  -A_{k}(z) \widetilde{A}_{k+1}(z) \\
&=\widetilde{A}_{k} (z)\{T_{-z_{k+1}}(z){A}_{k} (z)+ \overline{\Delta_{k+2}^{k+1}}
 {B}_{k} (z)\}-A_{k}(z)\{T_{-z_{k+1}}(z)\widetilde{A}_{k}(z)
 +\overline{\Delta_{k+2}^{k+1}}\widetilde{B}_{k}(z)\}\\
 &=\overline{\Delta_{k+2}^{k+1}}\{\widetilde{A}_{k}(z){B}_{k}(z)
 -{A}_{k}(z)\widetilde{B}_{k}(z)\}=\overline{\Delta_{k+2}^{k+1}}\prod_{\ell=1}^{k}T_{-z_{\ell}}(z) \prod_{\ell=0}^{k}(1-|\Delta_{\ell+1}^{\ell}|^2).
\end{align*}
\end{proof}
Correspondingly, we have the following observation.
\begin{lemma}\label{lemma:tildeBK+1-Bk}
For $k=0, \ldots , n$,
\begin{equation}\label{eq:tildeBK+1-Bk}
  B_{k}(z) \widetilde{B}_{k+1}(z)-\widetilde{B}_k(z)B_{k+1}(z)  =
  \Delta_{k+2}^{k+1}\prod_{\ell=0}^k\{T_{-z_{\ell+1}}(z) (1-|\Delta_{\ell+1}^{\ell}|^2)\}.
\end{equation}
\end{lemma}
\begin{proof}

It follows immediately from Lemma \ref{lemma:tilde-A-B-k} that
\begin{align*}
 &B_{k}(z) \widetilde{B}_{k+1}(z)-\widetilde{B}_{k}(z) B_{k+1}(z)  \\
 &=B_{k}(z)\{T_{-z_{k+1}}(z) \Delta_{k+2}^{k+1} \widetilde{A}_{k} (z)
 +\widetilde{B}_{k} (z)\}
 -\widetilde{B}_{k}(z)\{T_{-z_{k+1}}(z)\Delta_{k+2}^{k+1}  A_{k} (z)+  B_{k} (z)\}\\
 &=T_{-z_{k+1}}(z)\Delta_{k+2}^{k+1}\{\widetilde{A}_{k}(z){B}_{k} (z)-{A}_{k}(z)\widetilde{B}_{k}(z)\}=\Delta_{k+2}^{k+1}\prod_{\ell=0}^{k}\{T_{-z_{\ell+1}}(z) (1-|\Delta_{\ell+1}^{\ell}|^2)\}.
\end{align*}
\end{proof}

\begin{lemma}\label{lemma:relation-between-B-k-A-k}
For $k=0, \ldots , n $, the inequality
\begin{equation}\label{ineq:relation-between-B-k-A-k}
|B_k(z)|^2 - | A_k (z)|^2
\geq \prod_{\ell=0}^k (1-|\Delta_{\ell+1}^{\ell}|^2)
\end{equation}
holds for $z\in \overline{\mathbb D}$.
\end{lemma}
\begin{proof}
We prove this lemma by induction.
When $k=0$,  it follows from \eqref{eq:initial-A-B-0} that $|B_0(z)|^2-|A_0(z)|^2 = 1-|\Delta_1^0|^2$.
Assume \eqref{ineq:relation-between-B-k-A-k} holds for $k \geq 0$,
then for $|z| \leq 1$,  and  \eqref{eq:A-B-k-recurrence-formula} we have
\begin{align*}
   & |B_{k+1}(z)|^2  - |A_{k+1}(z)|^2\\
   =&
   |\Delta_{k+2}^{k+1} T_{-z_{k+1}}(z) A_k(z) + B_k(z) |^2
   -
   |T_{-z_{k+1}}(z) A_k(z) + \overline{\Delta_{k+2}^{k+1} }B_k(z)|^2
\\
  =&
   (1-|\Delta_{k+2}^{k+1}|^2) (|B_k(z)|^2  - |T_{-z_{k+1}}(z)|^2 |A_k(z)|^2 )
\\
  \geq&
   (1-|\Delta_{k+2}^{k+1}|^2) (|B_k(z)|^2  - |A_k(z)|^2 )
\\
  \geq&
   (1-|\Delta_{k+2}^{k+1}|^2) \prod_{\ell=0}^k (1-|\Delta_{\ell+1}^{\ell} |^2 )
   = \prod_{\ell=0}^{k+1} (1-|\Delta_{\ell+1}^{\ell} |^2 ) .
\end{align*}
\end{proof}

\begin{lemma}
For $k=0, \ldots , n$, the inequality
\begin{equation*}\label{ineq:tildeB-and-B-k}
|\widetilde{B}_k (z)| < |B_k(z)|
\end{equation*}
holds for $z\in \overline{\mathbb D}$.
\label{lemma:tildeB-and-B-k}
\end{lemma}
\begin{proof}
By Lemma \ref{lemma:relation-between-B-k-A-k}, the function $B_k(z)$ has no zeros on $\overline{\mathbb D}$.
Hence  $\widetilde{B}_k(z)/B_k(z)$ is analytic on $\overline{\mathbb D}$. For $|z|=1$,  using
Lemmas \ref{lemma:relation_between_A-k_B-k} and \ref{lemma:relation-between-B-k-A-k} we have
\begin{align*}
   |B_k(z)|^2 - |\widetilde{B}_k(z)|^2
    &=|B_k(z)|^2 - |\prod_{\ell=1}^k[T_{-z_{\ell}}(z)] \overline{A_k (1/\overline{z})}|^2\\
    &= |B_k(z)|^2 - |A_k(z)|^2
    \geq \prod_{\ell=0}^k (1-|\Delta_{\ell+1}^{\ell} |^2 ) > 0 .
\end{align*}
Thus we have $|\widetilde{B}_k(z)/B_k(z)| < 1$ on $\partial {\mathbb D}$,
and hence by the maximum modulus principle for analytic functions,
$|\widetilde{B}_k(z)/B_k(z)| < 1$ holds on $\overline{\mathbb D}$.
\end{proof}

\begin{lemma}\label{lem:interior_point}
Let $z \in {\mathbb D} \backslash \{ z_1,\ldots,z_{n+1} \}$.
Then $f_{\varepsilon}(z)|_{\varepsilon=0}$
is  an interior point of the variability region of $f(z)$.
\end{lemma}
\begin{proof}
For fixed $z\in {\mathbb D}$, note that $f_{\varepsilon }(z)$ defined by (\ref{def:extremal-f})
is an analytic function of $\varepsilon \in \overline{\mathbb D}$. Since a non-constant analytic function is an open map,
in order to prove the lemma it is sufficient to show that $f_{\varepsilon }(z)$ is a non-constant function of $\varepsilon$.

Let
$$
   \psi(z) =
   \left.
   \frac{\partial }{\partial \varepsilon} \left\{ f_{\varepsilon}(z ) \right\}
   \right|_{\varepsilon=0} .
$$
Then the problem reduces to showing  that $\psi(z) \not= 0$ for $z \in {\mathbb D} \backslash \{z_1,\ldots,z_{n+1} \}$.
By (\ref{eq:extremal-f-varepsilon}) and Lemma \ref{lemma:tilde-A-B-k}
we have
\begin{align*}
\psi(z)&=\frac{T_{-z_{n+1}}(z) \left\{ \widetilde{A}_n( z ) B_n( z ) - \widetilde{B}_n( z ) A_n( z ) \right\}}{B_n( z )^2}\\
&=\frac{\prod_{\ell=0}^n\{T_{-z_{\ell+1}}(z)(1-|\Delta_{\ell+1}^{\ell}|^2)\}}{B_n( z )^2} .
\end{align*}
We note that
  Lemma \ref{lemma:tildeB-and-B-k} implies
 $|B_n( z )|>0$, and so $\psi(z) $  has no zeros in ${\mathbb D} \backslash \{z_1,\ldots,z_{n+1} \}$.
\end{proof}

We now are able to state our first main result of this paper as follows.
\begin{theorem}[Multi-Point Schwarz-Pick Lemma]\label{thm:Main-Schwarz-Pick-theorem}
Fix pairwise distinct points $z_1,\ldots,z_{n+1}\in \mathbb{D}$ and corresponding interpolation values $w_1,\ldots,w_{n+1}\in \overline{\mathbb{D}}$.
Let all the hyperbolic divided differences
$\Delta_j^k \in {\mathbb D}$ for $0\le k<j\le n+1$.
Suppose that $f\in\ss$ such that
$f(z_j)=w_j$ for $0\le j\le n+1$.
Then for each fixed $z \in {\mathbb D} \backslash \{z_1,\ldots,z_{n+1} \}$,
 the region of values of $f(z)$ is the closed disk
$$ \overline{\D}(c(z),\rho(z))
  = \{  f_{\varepsilon}(z ) : \varepsilon \in \overline{\mathbb D} \},
$$
where
\begin{equation}
\left\{
\begin{aligned}
c (z) &=
   \frac{\overline{B_n(z)} \widetilde{B}_n(z) -|T_{-z_{n+1}}(z)|^2 \overline{A_n(z)} \widetilde{A}_n(z)}{ |B_n(z)|^2 -|T_{-z_{n+1}}(z)|^2 | A_n(z)|^2 } ,\\[2mm]
\rho(z)
    &=
     \frac{\prod_{k=0}^n \{|T_{-z_{k+1}}(z)|(1-|\Delta_{k+1}^k|^2)\} }{ |B_n(z)|^2 -|T_{-z_{n+1}}(z)|^2 | A_n(z)|^2 }.
\end{aligned}
\right.
\end{equation}
Furthermore
$$
 f(z)\in \partial \D(c(z),\rho(z))
$$
if and only if, $f \equiv f_{\varepsilon }$ for $\varepsilon \in \partial {\mathbb D}$.
\end{theorem}
To prove Theorem \ref{thm:Main-Schwarz-Pick-theorem} we need the following elementary lemma.
\begin{lemma}[\cite{Titchmarch}]
\label{lem:disk}
Let $p,q\in\mathbb{C}$, $k\neq 1$. The set $\{\lambda\in\mathbb{C}{:}|(\lambda-p)/(\lambda-q)|\leq k\}$ is the closed disk $\overline{\D}(c,\rho)$ with center and radius given by
$$
c=\frac{p-k^2q}{1-k^2},\quad \rho=\frac{|p-q|}{|1-k^2|}k.
$$
\end{lemma}

\begin{proof}[\textbf{Proof of Theorem \ref{thm:Main-Schwarz-Pick-theorem}}]
Let $f$ satisfies $f(z_k)=w_k$ for $k=1,\ldots,n+1$.
Then by Lemma  \ref{lemma:Schur_polynomial_and_Caratheodoy_problem},
there exists $f^* \in \mathcal{S}$ such that
$$
   f (z)  =
  \frac{T_{-z_{n+1}}(z) \widetilde{A}_n(z) f^* (z) + \widetilde{B}_n(z)}{T_{-z_{n+1}}(z)A_n(z)f^* (z)+B_n(z)}.
$$
Since $f^* \in \mathcal{S}$, we have
\begin{equation}\label{eq:f-region}
|T_{-z_{n+1}}(z) f^* (z)| =
    \left|  \frac{B_n(z)f (z) - \widetilde{B}_n(z) }{\widetilde{A}_n(z) - A_n(z) f (z)}  \right|
    \leq  |T_{-z_{n+1}}(z)| ,\quad z \in {\mathbb D},
\end{equation}
and equality holds if and only if $|f^* (z)|=1$.

It follows from Lemma \ref{lem:disk}
that the inequality (\ref{eq:f-region}) is equivalent to
$$
| f (z) - c (z) | \leq \rho (z) ,
$$
where
\begin{align*}
c (z) &=
   \frac{\overline{B_n(z)} \widetilde{B}_n(z) -|T_{-z_{n+1}}(z)|^2 \overline{A_n(z)} \widetilde{A}_n(z)}{ |B_n(z)|^2 -|T_{-z_{n+1}}(z)|^2 | A_n(z)|^2 } ,\\[2mm]
\rho(z) &=
    \frac{|T_{-z_{n+1}}(z)||B_n(z) \widetilde{A}_n(z) - A_n(z) \widetilde{B}_n(z)|}{ |B_n(z)|^2 -|T_{-z_{n+1}}(z)|^2 | A_n(z)|^2 }\\[2mm]
    &=
     \frac{\prod_{k=0}^n (1-|\Delta_{k+1}^k|^2)|T_{-z_{k+1}}(z)| }{ |B_n(z)|^2 -|T_{-z_{n+1}}(z)|^2 | A_n(z)|^2 }.
\end{align*}

Thus  $f (z) \in \overline{\mathbb D} ( c(z), \rho(z))$. Let $| \varepsilon | = 1$ and set
$$
f (z) = f_{\varepsilon}(z)
    = \frac{\varepsilon T_{-z_{n+1}}(z) \widetilde{A}_n(z)  + \widetilde{B}_n(z)}{\varepsilon T_{-z_{n+1}}(z)A_n(z)  + B_n (z) } .
$$
Hence by Lemma \ref{lemma:tilde-A-B-k}
 \begin{align}
&
f_{\varepsilon}(z) - c(z)
\label{eq:tilde_f-center}
\\
=&
\frac{\varepsilon  T_{-z_{n+1}}(z) \widetilde{A}_n(z)  + \widetilde{B}_n(z)}{\varepsilon T_{-z_{n+1}}(z)A_n(z) + B_n (z) }
    - \frac{\overline{B_n(z)} \widetilde{B}_n(z)-|T_{-z_{n+1}}(z) |^2 \overline{A_n(z)} \widetilde{A}_n(z)}{ |B_n(z)|^2 -|T_{-z_{n+1}}(z) |^2 | A_n(z)|^2 }
\nonumber
\\[2mm]
=&
   \frac{\varepsilon T_{-z_{n+1}}(z)  (B_n(z) \widetilde{A}_n(z) - A_n(z) \widetilde{B}_n(z) )}{ |B_n(z)|^2 -|T_{-z_{n+1}}(z) |^2 | A_n(z)|^2 }
   \cdot \frac{\; \overline{B_n(z)+ \varepsilon T_{-z_{n+1}}(z) A_n(z) } \; }{B_n(z)+ \varepsilon T_{-z_{n+1}}(z)  A_n(z)}
\nonumber
\\[2mm]
=&
   \frac{\varepsilon  \prod_{k=0}^n (1-|\Delta_{k+1}^k|^2)T_{-z_{k+1}}(z)  }{ |B_n(z)|^2 -|T_{-z_{n+1}}(z) |^2 | A_n(z)|^2 }
   \cdot \frac{|B_n(z)+ \varepsilon T_{-z_{n+1}}(z) A_n(z)|^2}{(B_n(z)+\varepsilon T_{-z_{n+1}}(z) A_n(z) )^2} .
\nonumber
\end{align}

In particular, $| f_{\varepsilon }(z) - c(z)| = \rho(z)$.
Thus for any $z \in {\mathbb D} \backslash \{ z_1,\ldots,z_{n+1}\}$, there exists $\theta \in {\mathbb R}$ such that
$f_{\varepsilon }(z) = c(z) + \rho(z)e^{i \theta }$.
\end{proof}

In particular, if we set $f_n(z)=\Delta_{n+1}^n$ in \eqref{eq:recurrence_formula_for-f}, then we obtain
$f(z)=\widetilde{B}_n(z)/{B}_n(z)$, which will be a Blaschke product of degree $n$ in the case $|\Delta_{n+1}^n|=1$.
\begin{theorem}
Fix pairwise distinct points $z_1,\ldots,z_{n+1}\in \mathbb{D}$ and corresponding interpolation values $w_1,\ldots,w_{n+1}\in \overline{\mathbb{D}}$.
Let all the hyperbolic divided differences
$\Delta_j^k \in {\mathbb D}$ for $0\le k<j\le n+1$, with the only exception that $|\Delta_{n+1}^n|=1$.
Suppose that $f\in\ss$ such that
$f(z_j)=w_j$ for $0\le j\le n+1$.
Then for each fixed $z \in {\mathbb D} \backslash \{z_1,\ldots,z_{n+1} \}$,
the set of values of $f(z)$ reduces to a set consisting of a single point
$$
w = T_{\Delta_1^0} ( T_{-z_1}(z) T_{\Delta_2^1}(\cdots T_{-z_{n-1}}(z) T_{\Delta_n^{n-1}} (\Delta_{n+1}^n T_{-z_n}(z)) \cdots ))).
$$
\end{theorem}

It is worth pointing out that the classical Schwarz-Pick Lemma is a simple corollary of Theorem \ref{thm:Main-Schwarz-Pick-theorem} for the case in which
the image of one point in $\D$ is known.
If the images of two points in $\D$ are known, then the so-called "two-point
Schwarz-Pick Lemma" (cf. \cite[Corollary 5.5]{kaptanoglu2002refine} and \cite[Proposition]{mercer1997sharpened}), which describes the range of values of $f$ at a third point $z\in\D$, can also be  directly derived from Theorem \ref{thm:Main-Schwarz-Pick-theorem}.
\begin{customthm}{B (Two-point Schwarz-Pick Lemma)}
Let pairwise distinct points $z_1,z_2\in\D$ and corresponding interpolation values $w_1,w_2\in\D$ with $\Delta_2^1\in \overline{\D}$. Suppose that $f\in \H$ such that  $f(z_1)=w_1$ and $f(z_2)=w_2$. Let $z\in\D\setminus\{z_1,z_2\}$.
\begin{enumerate}
\item If $|\Delta_2^1|=1$, then $f(z)=T_{\Delta_1^0} ( \Delta_2^1 T_{-z_1}(z))$.
\item If $|\Delta_2^1|<1$, then
 the range of values of $f(z)$ is the closed disk $\overline{\D}(c(z),r(z))$, where
\begin{align*}
c_2 (z) &=
   \frac{\overline{B_1(z)} \widetilde{B}_1(z) -|T_{-z_2}(z)|^2 \overline{A_1(z)} \widetilde{A}_1(z)}{ |B_1(z)|^2 -|T_{-z_2}(z)|^2 | A_1(z)|^2 } ,\\[2mm]
\rho_2(z) &=
     \frac{|T_{-z_2}(z)|^{2} (1-|w_1|^2)(1-|\Delta_2^1|^2) }{ |B_1(z)|^2 -|T_{-z_2}(z)|^2 | A_1(z)|^2 },
\end{align*}
and $A_1 (z)$, $\widetilde{A}_1(z)$,
  $B_1 (z)$, $\widetilde{B}_1 (z)$ are defined in \eqref{eq:A-B-1}.
In particular, if
 $z_1=w_1=0$, then
$$
c_2(z)=\frac{1-|T_{-z_2}(z)|^2}{1-|T_{-z_2}(z)|^2|w_2/z_2|^2}\frac{zw_2}{z_2},\quad \rho_2(z)=\frac{|zT_{-z_2}(z)|(1-|w_2/z_2|^2)}{1-|T_{-z_2}(z)|^2|w_2/z_2|^2}.
$$
\end{enumerate}
\end{customthm}

\section{Schur interpolation problem and generalized Rogosinski's Lemma}
Our method used to solve Problem \ref{prob:Schur} is similar to the Schur algorithm. For our purpose, we now give a new interpretation of the hyperbolic differences $\Delta_j^k$ to involve the higher-order hyperbolic derivatives.
For $n\ge 0$, we consider the specific case when $z_1=\cdots=z_{n+1}:=z_0$, $w_1=\cdots=w_{n+1}:=w_0$. Given the numbers $\gamma_1,\ldots,\gamma_n\in \D$, we define $\Delta_j^k=\gamma_k,0\le k< j\le n+1$.
For a given Schur data $\gamma =(\gamma_0,\ldots ,\gamma_n) \in {\mathbb D}^{n+1}$, we construct the functions $\Delta^j f(z)$, $j=0,\ldots,n+1$,
\begin{equation}\label{eq:def_f_k}
\Delta^kf(z) = \frac{[\Delta^{k-1}f(z) ,\Delta^{k-1}f(z_0)]}{[z,z_0]}.
\end{equation}
where $\Delta^{n+1}f\in \mathcal{S}$ is arbitrary.
This relation between $f$ and $\Delta^{n+1}f$ is invertible. By the definition of the higher-order hyperbolic derivatives, the equations in (\ref{eq:def_f_k}) show that
$\gamma_j = H^{(j)}f(z_0)=\Delta^jf(z_0)$ for $j=1,\dots,n$. Table \ref{table:generalized-Rogosinski-Lemma} describes the higher-order hyperbolic derivatives and the functions $\Delta^jf$.

\begin{table}[h]
\caption{
Table of hyperbolic divided differences for generalized Rogosinski's Lemma}
\label{table:generalized-Rogosinski-Lemma}
\begin{tabular}{cccccccccc}
\hline \multicolumn{2}{l}{ Points of $\mathbb{D}$} & 1 & 2 & 3 & $\ldots$ & $n-1$ & $n$ & $n+1$ \\
\hline $z_0$& $\gamma_0$ & & & & & & & \\
 & & $\gamma_{1}$ & & & & & & \\
 $z_0$ & $\gamma_0$ & & $\gamma_{2}$ & & & & & \\
 & & $\gamma_{1}$ & & $\gamma_{3}$ & & & & \\
 $z_0$ & $\gamma_0$ & & $\gamma_{2}$ & & $\ddots$ & & & \\
 & & $\gamma_{1}$ & & $\gamma_{3}$ & & $\gamma_{n-1}$ & & \\
 $z_0$ & $\gamma_0$ & & $\gamma_{2}$ & & & & $\gamma_{n}$ & \\
 & & $\gamma_{1}$ & & & & $\gamma_{n-1}$ & & $\Delta^{n+1}f(z)$ \\
 $z_0$ & $\gamma_0$ & & &$\vdots$ & & & $\Delta^n f(z)$ & \\
 & & & $\vdots$ & & & $\Delta^{n-1}f(z)$ & & \\
 {$\vdots$} & $\vdots$ & $\vdots$ & & & $.\cdot$ & & & \\
 & & & & $f_3(z)$ & & & & \\
$z_0$ & $\gamma_0$ & & $\Delta^2f(z)$ & & & & & \\
 &&$\Delta^1f(z)$&\\
$z$ & $f(z)$ &  & & & & & & \\
\hline
\end{tabular}
\end{table}

We give a solution to the Schur interpolation problem and  obtain an explicit representation of $f$ in terms of $\gamma$.

\begin{theorem}\label{thm:Schur}
Let the Schur parameter $\gamma=(\gamma_0, \ldots , \gamma_n) \in \overline{{\mathbb D}}^{n+1}$, $z_0\in \mathbb{D}$ and $f\in \mathcal{S}(\gamma)$. Then
\begin{enumerate}
\item
\label{interior}
If $|\gamma_0|<1, \ldots , |\gamma_n|<1$, then all solutions to the
Schur problem with  data $\gamma$ are given by
\begin{equation*}
f (z) =
T_{\gamma_0} (T_{-z_0}(z) T_{\gamma_1}(\cdots T_{-z_0}(z) T_{\gamma_n}(T_{-z_0}(z) f^* (z)) \cdots )),
\end{equation*}
where $f^* \in \mathcal{S}$ is arbitrary.

\item
\label{boundary}
If $|\gamma_0|<1, \ldots , |\gamma_{j-1}|<1$, $|\gamma_j|=1$, $\gamma_{j+1}= \cdots = \gamma_n=0$
for some $j=0, \ldots , n$, then the Schur problem with  data $\gamma$ has the unique solution
\begin{equation*}
f (z) =
  T_{\gamma_0} (T_{-z_0}(z) T_{\gamma_1}(\cdots T_{-z_0}(z) T_{\gamma_{j-1}}(\gamma_j T_{-z_0}(z))  \cdots )),
\end{equation*}
which is a Blaschke product of degree $j$.
\item
\label{exterior}
In all other cases,
there is no solution to the Schur problem with  data $\gamma$.
\end{enumerate}
\end{theorem}

\begin{proof}
(1) Let the Schur parameter $\gamma =( \gamma_0, \ldots , \gamma_n )\in \mathbb{D}^{n+1}$ , i.e.,
$|\gamma_0| < 1, \ldots , |\gamma_n| < 1$. We will show that the Schur interpolation problem has infinitely many solutions, and for any function $f^*\in \mathcal{S}$ there exists a (unique) solution $f\in \ss$ with $\Delta^{n+1}f=f^*$. The idea is to find $f$ from $\Delta^{n+1}f$ using the inversion formula from \eqref{eq:def_f_k}. To do this, we start from $f_{n+1}:=f^*\in \ss$ and define $f_n, \ldots , f_0$ inductively by
\begin{equation}\label{eq:f-k-chain}
f_{k}(z):=\big[[z,z_0]\cdot f_{k+1}(z),-\gamma_k\big]=T_{\gamma_k}(T_{-z_0}(z)f_{k+1}(z)),\quad k=0,\ldots,n.
\end{equation}
Then $f_k(z_0)=[0,-\gamma_k]=\gamma_k$ and therefore $\Delta_{z_0}f_k=f_{k+1}$.

It remains to prove that $H^kf_0(z)=\gamma_k$, for $1\leq k\leq n$.  For this purpose, we first show that $\Delta_{z_0}^kf_0(z) = f_k(z)$, for $z\in \mathbb{D} $. We use induction on $k$. For $k= 1$, we have
$$
\begin{aligned}
\Delta_{z_0}^{1}f_0(z)& =\frac{[f_0(z),\:f_0(z_0)]}{[z,z_0]}  \\
&=\frac{\big[[[z,z_0]\cdot f_1(z),-\gamma_0],\gamma_0\big]}{[z,z_0]} \\
&=f_1(z).
\end{aligned}
$$
In particular, we have $H^1f_0(z_0)=\Delta_{z_0}^1f_0(z_0)=f_1(z_0)$. Suppose that the assertion holds for some $k$, i.e., $\Delta_{z_0}^kf_0(z)=f_k(z)$ and  $H^kf_0(z_0)=\Delta_{z_0}^kf_0(z_0)=f_k(z_0)=\gamma_k$. Hence
$$
\begin{aligned}
\Delta_{z_0}^{k+1}f_0(z)& =\frac{[\Delta_{z_0}^kf_0(z),\Delta_{z_0}^kf_0(z_0)]}{[z,z_0]} \\
&=\frac{[f_k(z),f_k(z_0)]}{[z,z_0]} \\
&=\frac{\big[[[z,z_0]\cdot f_{k+1}(z),-\gamma_k],\gamma_k\big]}{[z,z_0]} \\
&=f_{k+1}(z).
\end{aligned}
$$
 Finally, for $1\leq k\leq n$, we have
$H^kf_0( z_0) =f_k(z_0)= \gamma_k$.
We now set $f:=f_0$, then $\Delta^kf(z)=f_k(z)$. Therefore, starting with $f_{n+1}\in \mathcal{S}$ we obtain the function $f_0$, which solves the interpolation problem.

\bigskip

(2) Suppose that $|\gamma_0|<1, \ldots , |\gamma_{j-1}|<1$, $|\gamma_j|=1$, $\gamma_{j+1}= \cdots = \gamma_n=0$
for some $j=0, \ldots , n$. Then $|\Delta^jf(z)|=|\gamma_j|=1$ and by the maximum modulus principle, $\Delta^jf$ is the unimodular constant $\gamma_j$. Thus, $f$ must be a Blaschke product of degree $j$. In this case, $\Delta^k f=0$ for $k>j$. Define
$$f_{k}(z):=\big[[z,z_0]\cdot f_{k+1}(z),-\gamma_k\big]=T_{\gamma_k}(T_{-z_0}(z)f_{k+1}(z)),\quad k=0,\ldots,j-1.\,$$
where $f_j=\gamma_j$. We can obtain the explicit form of $f$ from this construction, which is a unique solution to the Schur interpolation problem.

\bigskip

(3) From Case (1) and (2), we know that there exists no solution to the Schur interpolation problem in any other cases.
\end{proof}

From now on, we will assume that the Schur parameter $\gamma =( \gamma_0,\ldots , \gamma_n) \in {\mathbb D}^{n+1}$, unless otherwise stated.
In the proof of Case (1) of Theorem \ref{thm:Schur}, we shows that $f_k$ coincides with $\Delta^k f(z)$ for $j=0,\ldots,n+1$. Therefore, we use $f_k$ instead of $\Delta^k f(z)$.
By (\ref{eq:def_f_k}), we have defined $f_1, \ldots , f_{n+1}
\in \mathcal{S}$ by
$$
f_{k+1} = \frac{[f_k(z) ,f_k(z_0)]}{[z,z_0]},\quad k=0,\ldots,n.
$$
Since $\gamma_k = f_k(z_0)$, $k=0, \ldots , n$, we have
\begin{equation}\label{eq:recurrence_formula_for_f_k}
f (z ) =  f_0 (z) \quad\mbox{and}\quad
f_k (z ) = \frac{T_{-z_0}(z) f_{k+1} (z)+ \gamma_k }{1+ \overline{\gamma_k }T_{-z_0}(z) f_{k+1} (z)},\quad (k=0,1,\ldots,n).
\end{equation}

Using a similar technique in Section 2,
we define sequences of rational functions $A_k (z), B_k (z), \widetilde{A}_k (z)$ and $\widetilde{B}_k(z)$ recursively by
\begin{equation}\label{eq:initial-A-B_0}
 \begin{pmatrix}
  A_0 (z) & \widetilde{A}_0(z) \\
  B_0 (z) & \widetilde{B}_0 (z) \\
 \end{pmatrix}
 =
 \begin{pmatrix}
  \overline{\gamma_0} & 1 \\
  1 & \gamma_0 \\
 \end{pmatrix}
\end{equation}
and
\begin{equation}\label{eq:recurrence_formula}
 \begin{pmatrix}
  A_{k+1} (z) & \widetilde{A}_{k+1}(z) \\
  B_{k+1} (z) & \widetilde{B}_{k+1} (z) \\
 \end{pmatrix}
 =
 \begin{pmatrix}
  T_{-z_0}(z) & \overline{\gamma_{k+1}} \\
  \gamma_{k+1}T_{-z_0}(z) & 1 \\
 \end{pmatrix}
 \begin{pmatrix}
  A_k (z) & \widetilde{A}_k (z) \\
  B_k (z) & \widetilde{B}_k (z) \\
 \end{pmatrix} ,
\end{equation}
where $k=0,\ldots , n-1 $.
 From (\ref{eq:initial-A-B_0}) and
(\ref{eq:recurrence_formula}) it  follows easily that
\begin{equation}\label{eq:A-1}
 \begin{pmatrix}
  A_1 (z) & \widetilde{A}_1(z) \\
  B_1 (z) & \widetilde{B}_1 (z) \\
 \end{pmatrix}
 =
 \begin{pmatrix}
  \overline{\gamma_{1}}+\overline{\gamma_{0}}  T_{-z_{0}}(z) & \gamma_{0} \overline{\gamma_1}+T_{-z_0}(z) \\
  1+\overline{\gamma_0}\gamma_{1}T_{-z_{0}}(z) & \gamma_{0}+\gamma_{1}  T_{-z_{0}}(z)  \\
 \end{pmatrix},
\end{equation}
and
\begin{equation}
\left\{
\begin{aligned}
A_{2}(z)&=\overline{\gamma_2}+\left(\overline{\gamma_1}+\overline{\gamma_0} \gamma_{1} \overline{\gamma_2}\right) T_{-z_{0}}(z)+\overline{\gamma_0}T_{-z_0}(z)^{2}, \\
\widetilde{A}_2(z)&=\gamma_{0} \bar{\gamma}_{2}+\left(\gamma_{0} \overline{\gamma_{1}}+\gamma_{1} \bar{\gamma}_{2}\right)T_{-z_0}(z)+\left[T_{-z_0}(z)\right]^{2},
\\
B_{2}(z)&=1+\left(\overline{\gamma_0} \gamma_{1}+\overline{\gamma_1} \gamma_{2}\right) T_{-z_0}(z)+\overline{\gamma_0} \gamma_{2}T_{-z_{0}}(z)^{2}, \\
\widetilde{B}_2 (z)&= \gamma_{0}+\left(\gamma_{1}+\gamma_{0} \overline{\gamma_1} \gamma_{2}\right)T_{-z_{0}}(z)+\gamma_{2}\left[T_{z_0}(z)\right]^{2}.  \\
\end{aligned}
\right.
\end{equation}
Also
\begin{equation}\label{eq:B_of_zero}
B_k(z_0) = 1 , \quad \text{and} \quad  \widetilde{B}_k(z_0) = \gamma_0
\end{equation}
for $k=0,1, \ldots , n$.

By (\ref{eq:recurrence_formula_for_f_k}) and induction, we have the following recurrence formula for $f(z)$,
\begin{equation}\label{eq:recurrence_formula_for_f}
      f (z) =
      \frac{T_{-z_0}(z) \widetilde{A}_k(z) f_{k+1} (z) + \widetilde{B}_k(z)}
      {T_{-z_0}(z)A_k(z) f_{k+1} (z) + B_k(z)}
      , \quad
      k=0,1, \ldots , n .
\end{equation}
In particular, if we set $f_n(z)=\gamma_n$ in \eqref{eq:recurrence_formula_for_f}, then we obtain
$f(z)=\widetilde{B}_n(z)/{B}_n(z)$, which will be a Blaschke product of degree $n$ in the case $|\gamma_n|=1$.

For $\varepsilon \in \overline{\mathbb D}$, let
\begin{align}
f_{\gamma , \varepsilon }(z)
    =&  T_{\gamma_0} ( T_{-z_0}(z) T_{\gamma_1} ( \cdots T_{-z_0}(z) T_{\gamma_{n}} ( \varepsilon T_{-z_0}(z)) \cdots )), \quad z \in {\mathbb D}.
\label{def:extremal_f}
\end{align}
Then $f_{\gamma , \varepsilon } \in \mathcal{S}$ is a solution to the Schur problem, i.e., $f_{\gamma , \varepsilon } \in \mathcal{S}(\gamma)$.
We note that for each fixed $\varepsilon \in \overline{\mathbb D}$, $f_{\gamma,\varepsilon} (z)$ is analytic functions of $z \in {\mathbb D}$.
In particular, $f_{\gamma,\varepsilon} (z)$ is a finite Blaschke product of $z$ for $|\varepsilon | =1$.

We next give some important properties of $A_k (z), B_k (z), \widetilde{A}_k (z)$ and $\widetilde{B}_k(z)$ without detailed proofs, which are corresponding to the lemmas in Section 2.
\begin{lemma}\label{lemma:Schur_polynomial_and_Caratheodoy_problem}
Let  the Schur parameter $\gamma=(\gamma_0,\ldots,\gamma_{n})\in {\mathbb D}^{n+1}$. Then for any $f\in \mathcal S(\gamma)$, there exists a unique $f^{*}\in \mathcal S$ such that
\begin{equation}\label{eq:SChur_representation}
f (z) = \frac{T_{-z_0}(z)\widetilde{A}_n(z) f^*(z)+ \widetilde{B}_n(z)}{T_{-z_0}(z)A_n(z)f^*(z) + B_n(z)} .
\end{equation}
Conversely, for any $f^{*}\in \mathcal S$, if  $f$ is given by $(\ref{eq:SChur_representation})$, then $f\in \mathcal S(\gamma)$. In particular, the function $f_{\gamma,\varepsilon}$ defined by (\ref{def:extremal_f}) can be written as
\begin{equation}\label{eq:extremal_f_and_schur_polynomials}
f_{\gamma , \varepsilon } (z)
   = \frac{ \varepsilon T_{-z_0}(z)\widetilde{A}_n(z)+ \widetilde{B}_n(z)}{ \varepsilon T_{-z_0}(z) A_n(z) + B_n(z)},
\end{equation}
and $f_{\gamma,\varepsilon}\in{\mathcal S}(\gamma)$.
\end{lemma}

\begin{lemma}\label{lemma:relation_between_Moebius_coefficients}
For $k=0, \ldots , n$,
\begin{align*}
  \widetilde{A}_k (z) &= [T_{-z_0}(z)]^k \overline{B_k (1/\overline{z})} ,
   \quad
   \widetilde{B}_k (z) = [T_{-z_0}(z)]^k \overline{A_k (1/\overline{z})},\\
{A}_k (z) &= [T_{-z_0}(z)]^k\overline{\widetilde{B}_k (1/\overline{z})},
   \quad
   B_k (z) =[T_{-z_0}(z)]^k \overline{\widetilde{A}_k (1/\overline{z})}.
\end{align*}
\end{lemma}

\begin{lemma}
For $k=0, \ldots , n$,
$$
  \widetilde{A}_k (z) B_k(z) - A_k(z)  \widetilde{B}_k(z) = [T_{-z_0}(z)]^k \prod_{\ell=0}^k (1-|\gamma_\ell|^2) .
$$
\end{lemma}
\begin{lemma}
For $k=0, \ldots , n$,
\begin{align*}
&\widetilde{A}_k (z) \overline{\widetilde{A}_k (1/\overline{z})}
- A_k(z)  \overline{A_k (1/\overline{z})} = \prod_{\ell=0}^k (1-|\gamma_{\ell}|^2),\\
&B_k(z)\overline{B_k (1/\overline{z})} -  \widetilde{B}_k(z)\overline{\widetilde{B}_k (1/\overline{z})} = \prod_{\ell=0}^k (1-|\gamma_{\ell}|^2).
\end{align*}
\end{lemma}
\begin{lemma}
For $k=0, \ldots , n$,
\begin{align*}
  \widetilde{A}_k(z)A_{k+1}(z)  -A_{k}(z) \widetilde{A}_{k+1}(z)&=
  \overline{\gamma_{k+1}}[T_{-z_{0}}(z)]^k\prod_{\ell=0}^k (1-|\gamma_{\ell}|^2),\\
B_{k}(z) \widetilde{B}_{k+1}(z)-\widetilde{B}_{k}(z)B_{k+1}(z) &=\gamma_{k+1}[T_{-z_0}(z)]^{k+1} \prod_{\ell=0}^{k}\left(1-|\gamma_{\ell}|^{2}\right).
\end{align*}
\end{lemma}
\begin{lemma}
For $k=0, \ldots , n $, the inequality
$$
|B_k(z)|^2 - | A_k (z)|^2
\geq \prod_{\ell=0}^k (1-|\gamma_\ell|^2)
$$
holds for $z\in \overline{\mathbb D}$.
\end{lemma}

\begin{lemma}
For $k=0, \ldots , n$, the inequality
$
|\widetilde{B}_k (z)| < |B_k(z)|
$
holds for $z\in \overline{\mathbb D}$.
\end{lemma}

\begin{lemma}\label{prop:interior_point}
Let the Schur parameter $\gamma =(\gamma_0,\ldots ,\gamma_n) \in \mathbb{D}^{n+1}$ and $z \in {\mathbb D} \backslash \{ z_0 \}$.
Then $f_{\gamma,0}(z)$
is  an interior point of the set $V(z, \gamma )$.
\end{lemma}

Now we are in a position to show our main results as follows.

\begin{theorem}[Rogosinski-Pick Lemma for higher-order hyperbolic derivatives]\label{thm:genelized-Rogosinski-Pick-thm}
Let $n \in {\mathbb N}$, $z_0\in \mathbb{D}$ and the Schur parameter
$\gamma  = (\gamma_0, \ldots , \gamma_n ) \in {\mathbb D}^{n+1}$. Suppose that $f\in \mathcal S(\gamma)$.
Then for each fixed $z \in {\mathbb D} \backslash \{z_0 \}$, the region of values of $f(z)$ is the closed disk
$$
V(z,\gamma) = \{ f_{\gamma , \varepsilon}(z) : \varepsilon \in \overline{\mathbb D} \}=\D(c(z),\rho(z)),
$$
where
\begin{align*}
c (z) &=
   \frac{\overline{B_n(z)} \widetilde{B}_n(z) -|T_{-z_0}(z)|^2 \overline{A_n(z)} \widetilde{A}_n(z)}{ |B_n(z)|^2 -|T_{-z_0}(z)|^2 | A_n(z)|^2 } ,\\[2mm]
\rho(z) &=
     \frac{|T_{-z_0}(z)|^{n+1} \prod_{k=0}^n (1-|\gamma_k|^2) }{ |B_n(z)|^2 -|T_{-z_0}(z)|^2 | A_n(z)|^2 }.
\end{align*}
Furthermore
$$
 f(z)\in \partial \D(c(z),\rho(z))
$$
if and only if, $f \equiv f_{\gamma ,\varepsilon }$ for $\varepsilon \in \partial {\mathbb D}$.
\end{theorem}

\begin{remark}
We claim that $\rho(z)\to 0$ as $n\to \infty$. Indeed, for $z\in \D\setminus\{z_0\}$,
$$|B_n(z)|^2 -|T_{-z_0}(z)|^2 | A_n(z)|^2\ge |B_n(z)|^2 -| A_n(z)|^2\ge \prod_{k=0}^n (1-|\gamma_k|^2).$$
Therefore,
$$
\rho(z) =
     \frac{|T_{-z_0}(z)|^{n+1} \prod_{k=0}^n (1-|\gamma_k|^2) }{ |B_n(z)|^2 -|T_{-z_0}(z)|^2 | A_n(z)|^2 }
     \le |T_{-z_0}(z)|^{n+1},$$
     where $|T_{-z_0}(z)|^{n+1}\rightarrow 0$ as $n\to \infty$.
\end{remark}

It is clear that for a given Schur parameter $\gamma $,
the hypotheses of (\ref{boundary}) and \eqref{exterior}
in Theorem \ref{thm:Schur}  provide a straightforward solution to Problem \ref{the_problem}.

\begin{theorem}
Let  $\gamma  =(\gamma_0, \ldots , \gamma_n) \in \overline{{\mathbb D}}^{n+1}$ be a Schur parameter. If the hypothesis of {\rm (\ref{boundary})} in Theorem \ref{thm:Schur} holds,
then $V(z,\gamma )$ reduces to a set consisting of a single point
$$
w= T_{\gamma_0} ( T_{-z_0}(z) T_{\gamma_1}(\cdots T_{-z_0}(z) T_{\gamma_{j-1}} (\gamma_j T_{-z_0}(z)  ) \cdots )),
$$
and $\gamma =( \gamma_0, \ldots , \gamma_j, 0, \ldots , 0 )$ is  the Schur parameter.

If the hypothesis of {\rm (\ref{exterior})} in Theorem \ref{thm:Schur} holds,
then $V(z,\gamma ) = \emptyset$.
\end{theorem}

It is worth noting that, for the Schur parameter
$\gamma=(\gamma_0)\in \D$, the classical Schwarz-Pick Lemma is a straightforward corollary of Theorem \ref{thm:genelized-Rogosinski-Pick-thm}. Furthermore,
Theorem \ref{thm:genelized-Rogosinski-Pick-thm} asserts that for $f\in\ss$ and $z_0\in\D$,
 depending on the hyperbolic derivatives up to a certain
order at $z_0$, we can determine the variability region of $f(z)$ for any point $z\in \D$. We will see that Rogosinski's Lemma is also a corollary of Theorem \ref{thm:genelized-Rogosinski-Pick-thm}. For instance, let $z_0=\gamma_0=0$, $f\in \H$ and $H^1 f(0)=\gamma_1$. From \eqref{eq:A-1}, we obtain that $$A_1(z)=\overline{\gamma_1},\widetilde {A_1}(z)=z,B_1(z)=1,\widetilde {B_1}(z)=\gamma_1 z.$$
By Theorem \ref{thm:genelized-Rogosinski-Pick-thm}, the range of values of $f(z)$ is the closed disk $\overline{\D}(c(z),\rho(z))$, where
 $$c(z)=\dfrac{z\gamma_1(1-|z|^2)}{1-|z|^2|\gamma_1|^2},\quad
\rho(z)= \dfrac{|z|^2(1-|\gamma_1|^2)}{1-|z|^2|\gamma_1|^2},$$
which is Rogosinski's Lemma since $f'(0)=\gamma_1$.

We can avoid the restriction $f(0) = 0$ and obtain an analogous version, which is called the Rogosinski-Pick Lemma,  for any point of the unit disk  different from $z_0$ (see \cite{kaptanoglu2002refine}, \cite{rivard2013application}).

\begin{customthm}{D (Rogosinski-Pick Lemma)}
Let $z_0, \gamma_0\in \D$ and $\gamma_1\in \overline{\D}$. Suppose that $g\in \H$, $f(z_0) = \gamma_0$ and $H^1 f(z_0)=\gamma_1$. Let $z\in\D\setminus\{ z_0\}$.
 \begin{enumerate}
\item If $|\gamma_1|=1$, then $f(z)=T_{\gamma_0} ( \gamma_1 T_{-z_0}(z))$.

\item If $|\gamma_1|<1$, then the range of values of $f(z)$ is the closed disk $\overline{\D}(c(z),\rho(z))$, where
\begin{align*}
c (z) &=
   \frac{\overline{B_1(z)} \widetilde{B}_1(z) -|T_{-z_0}(z)|^2 \overline{A_1(z)} \widetilde{A}_1(z)}{ |B_1(z)|^2 -|T_{-z_0}(z)|^2 | A_1(z)|^2 } ,\\[2mm]
\rho(z) &=
     \frac{|T_{-z_0}(z)|^{2} (1-|\gamma_0|^2)(1-|\gamma_1|^2) }{ |B_1(z)|^2 -|T_{-z_0}(z)|^2 | A_1(z)|^2 },
\end{align*}
and $A_1 (z)$, $\widetilde{A}_1(z)$,
  $B_1 (z)$, $\widetilde{B}_1 (z)$ are defined in \eqref{eq:A-1}.
In particular, if
 $\gamma_0=0$, then
$$
c(z)=\frac{\gamma_1T_{-z_0}(z)(1-|T_{-z_0}(z)|^2)}{1-|T_{-z_0}(z)|^2|\gamma_1|^2}\quad\mathrm{and}\quad \rho(z)=\frac{|T_{-z_0}(z)|^2(1-|\gamma_1|^2)}{1-|T_{-z_0}(z)|^2|\gamma_1|^2}.
$$
\end{enumerate}
\end{customthm}

Finally, we end this section with the following invariance property.
\begin{lemma}
Let $n \in {\mathbb N}$, $z_0,z\in \mathbb{D}$ and the Schur parameter
$\gamma  = (\gamma_0, \ldots , \gamma_n ) \in {\mathbb D}^{n+1}$. Suppose that $f_0=f\in \mathcal S(\gamma)$ and
$f_{j+1}(z)=[f_j(z),\gamma_j]/T_{-z_0}(z)$, $\gamma_j=f_j(z_0)$ for $j=0,1,...,n$. Let
$$g_0(z)=g(z)=f(\frac{z+z_0}{1+\overline{z_0}z}),\quad g_{j+1}(z)=\frac{[g_j(z),\delta_j]}{z},$$
and $\delta_j=g_j(0)$ for $j=0,1,...,n$.
Then
$f_j(z)=g_j \circ T_{-z_0}(z)$ and $\gamma_j=\delta_j$ for $j=0,1,...,n$.
\end{lemma}
\begin{proof}
We proof this lemma by induction. For $j=0$, it is obvious that $f_0(z)=g_0 \circ T_{-z_0}(z)$ and $\gamma_0=f_0(z_0)=g_0(0)=\delta_0$.

For $j=1$, from the chain rule, we have $\Delta_{z_0}(g\circ T_{z_0})=(\Delta_{T_{-z_0}(z_0)}g)\circ T\cdot \Delta_{z_0}T_{-z_0}$.
We note that $T_{-z_0}(z_0)=0$ and $$\Delta_{z_0}T_{-z_0}(z)=\frac{[T_{-z_0}(z),T_{-z_0}(z_0)]}{[z,z_0]}
=\frac{[T_{-z_0}(z),0]}{[z,z_0]}=1.$$
Thus $\Delta_{z_0}(g\circ T)=(\Delta_0 g)\circ T_{-z_0}$, which is $f_1=g_1 \circ T_{-z_0}$. In particular, $f_1(z_0)=g_1(T_{-z_0}(z_0))=g_1(0)=\delta_1$.

We suppose that for $j\le k$, $f_j=g_j \circ T_{-z_0}$ and $\gamma_j=\delta_j$, then for $j=k+1$, by induction,
\begin{align*}
  f_{k+1}(z)&=\Delta_{z_0}f_k(z)=\Delta_{z_0}(g_k \circ T_{-z_0})(z)\\
  &=(\Delta_{T_{-z_0}(z_0)}g_k)\circ T_{-z_0}(z)\cdot \Delta_{z_0}T_{-z_0}(z)\\
  &=(\Delta_{0}g_k)\circ T_{-z_0}(z)\\
  &=g_{k+1}\circ T_{-z_0}(z).
\end{align*}
In particular, we have $f_{k+1}(z_0)=g_{k+1}(T_{-z_0}(z_0))=\delta_{k+1}$.
\end{proof}
\section*{Acknowledgement}
The author would like to express his deep gratitude to Prof. Toshiyuki Sugawa for his valuable comments and instructive suggestions.

\end{document}